\documentclass[a4paper, 11pt]{amsart}

\usepackage{amsthm, amsfonts, amsmath, amssymb,amstext}
\usepackage{enumerate}			
\usepackage[hang]{subfigure} 
\usepackage{tikz}           
\usepackage{caption}
\usepackage{ifthen}					

\newtheorem{thm}{Theorem}[section]
\newtheorem{lem}[thm]{Lemma}
\newtheorem{prop}[thm]{Proposition}
\newtheorem{cor}[thm]{Corollary}
\newtheorem{open}[thm]{Open Problem}
\theoremstyle{definition}
\newtheorem{defn}[thm]{Definition}
\newtheorem{claim}[thm]{Claim}

\newtheorem{example}[thm]{Example}
\newtheorem*{cor:InductiveConstruction_TreesAndGraph}{Corollary \ref{cor:InductiveConstruction_TreesAndGraph}}
\newtheorem*{cor:RealisationIn2Dimensions}{Corollary \ref{cor:RealisationIn2Dimensions}}

\theoremstyle{remark}
\newtheorem{case}{Case}[thm] 

\usetikzlibrary{arrows, backgrounds}
\tikzset{vertex/.style={circle,fill=black,minimum size=0.4pt,inner sep=1.8pt}}
\tikzset{emptyvertex/.style={shape=circle, inner sep = 1.8pt, minimum size=0.4pt, draw}}
\tikzset{extended line/.style={shorten >=-#1,shorten <=-#1},
 extended line/.default=8cm}
\newcommand{\eq}{=} 
\newcommand{\bR}{{\mathbb{R}}}

\newcommand{\bZ}{{\mathbb{Z}}}

\def\C{{\mathcal{C}}}

 \def\G{{\mathcal{G}}}

 \newcommand{\K}{{\mathcal{K}}}

 \newcommand{\Aut}{\text{Aut}}
 \newcommand{\GL}{\text{GL}}

\begin{document}
\title{Constructing isostatic frameworks for the $\ell^\infty$ plane}


\author{K.~Clinch}%
\address{School of Mathematical Sciences, Queen Mary University of London,  Mile End Road, London, E1 4NS, UK} 
\email{k.clinch@qmul.ac.uk}%
\thanks{The first named author was supported by an Engineering and Physical Sciences Research Council studentship.}

\author{D.~Kitson}%
\address{Department of Mathematics and Statistics, Lancaster University, Lancaster, LA1 4YF, UK}
\email{d.kitson@lancaster.ac.uk}
\thanks{The second named author was supported by the Engineering and Physical Sciences Research Council [grant number EP/P01108X/1].}


\begin{abstract}
We use a new {coloured} multi-graph constructive method to prove that every 2-tree decomposition can be realised in the plane as a bar-joint framework which is minimally rigid (isostatic) with respect to $\ell^1$ or $\ell^\infty$ distance constraints. We show how to adapt this technique to incorporate symmetry and indicate several related open problems on rigidity, redundant rigidity and forced symmetric rigidity in normed spaces.    
\end{abstract}

\maketitle
\section{Introduction}
A simple graph $G=(V,E)$ with vertices embedded generically in $\bR^2$ inherits a natural edge-labelling $\kappa:E\to\{1,2\}$ whereby an edge  (represented by a straight line segment between its embedded vertices) is labelled either $1$ or $2$ depending on whether the slope $m$ of its affine span satisfies $|m|<1$ or $|m|>1$. Simple examples show that not all edge-labellings $\kappa:E\to\{1,2\}$ are realisable in this way. 
Motivated by problems in  graph rigidity under $\ell^p$ distance constraints (see \cite{KP_NonEuclideanNorms,KS_IncidentalSymm,KS_Reflection}), we are interested in the realisability of $d$-tree decompositions in $\bR^d$ (see Section \ref{Preliminaries} for the corresponding definition when $d>2$). 
A {\em $d$-tree decomposition} arises from an edge-labelling $\kappa:E\to\{1,2,\ldots,d\}$ when the edge sets $\kappa^{-1}(1),\ldots,\kappa^{-1}(d)$ are spanning trees in $G$.
In general, multi-graphs which admit a $d$-tree decomposition 
are characterised by the conditions $|E|=d(|V|-1)$ and $|E(H)|\leq d(|V(H)|-1)$ for each subgraph $H$ (see Nash-Williams \cite{N-W_SpanningTrees} and Tutte \cite{T_SpanningTrees}) and such graphs are said to be {\em $(d,d)$-tight}.
Constructive characterisations for  $(d,d)$-tight graphs and connections to graph rigidity under $\ell^2$ distance constraints are discussed in Tay \cite{Tay_Multigraphs1}, Frank and Szeg\"{o} \cite{FS_PackingAndCoveringTrees} and in Graver, Servatius and Servatius \cite[\S4.9]{GSS_CombinatorialRigidity}, for example. 

It is shown in \cite{KP_NonEuclideanNorms} that, for $p\not=2$, the structure graph of an isostatic bar-joint framework (see Section \ref{Preliminaries}) in the $\ell^p$ plane is necessarily $(d,d)$-tight. However, the validity of the converse statement is an open problem for $d\geq 3$ and determining the realisability or otherwise of $d$-tree decompositions in $\bR^d$ would settle this conjecture for the $\ell^\infty$ context. 
In this article we present a constructive method for realising {coloured} 2-tree decompositions in the plane. The method consists of two parts: a multi-graph construction scheme for $(d,d)$-tight graphs which tracks the evolution of $d$ edge-disjoint spanning trees, and in the case $d=2$, a method of constructing geometric placements for these multi-graphs which accommodates parallel edges. While it is known that $(d,d)$-tight graphs are constructible in terms of multi-graphs, the particular role of spanning trees in these constructions is given prominence here. Moreover, the method of assigning geometric placements to multi-graphs used here is a new technique which can be adapted for other contexts. The graph construction is presented in Section \ref{sec:AllD_Combinatorics} and the geometric realisations are contained in Section \ref{sec:2D_Geometry}. In Section 
\ref{sec:Symmetry}, we illustrate the versatility of the technique by adapting it to symmetric $2$-tree decompositions. In the concluding section we state some related open problems and indicate connections to other areas of graph rigidity for normed spaces. 

This work was initiated at the workshop ``Advances in Combinatorial and Geometric Rigidity" which took place at the Banff International Research Station from 12th-17th July 2015.

\section{Preliminaries}\label{Preliminaries} 
In this section we introduce terminology, state the main results and provide some background.

\subsection{Graph Theory}

Let $G=(V,E)$ be a finite, loop-free multi-graph. Let $X\subseteq	 V$ be a set of vertices. The \emph{neighbourhood} of $X$ in $G$, $N_G(X)$, is the set of all vertices in $V-X$ which share an edge with some $x\in X$. When $X=\{x\}$ we refer to $N_G(x)$ instead of $N_G(\{x\})$. The subgraph of $G$ \emph{induced} by $X$ is denoted $G[X]$ and has vertex set $X$ and edge set $E_G(X) = \{uv\in E: u,v\in X\}$. We let $i_G(X)= |E_G(X)|$. When the original graph $G$ is apparent from the context, we omit the subscripts and refer simply to $N(X)$, $E(X)$ and $i(X)$.
For $F\subseteq E$, the subgraph of $G$ \emph{induced} by $F$ is denoted $G[F]$ and has vertex set $V_G(F) = \{v\in V: uv\in F \text{ for some } u\in V\}$, and edge set $F$. 
We say that the edge set $F$ \emph{spans} $G$ when $V_G(F)=V$. Once more, when the graph is apparent, we omit the subscripts.
For a vertex $v\in V$, the \emph{degree} of $v$ is the number of edges incident {to} $v$, and is denoted by $d_G(v)$ or $d(v)$. Note that for multi-graphs $d_G(v) \geq |N_G(v)|$. The \emph{minimum degree of $G$} is the minimum of $d(v)$ for all $v\in V$ and is denoted $\delta(G)$.

A {\em $d$-tree decomposition} is a tuple $\G=(G;T_1,\ldots,T_d)$ where $G$ is a multi-graph which is the edge-disjoint union of spanning trees $T_1,\ldots,T_d$. {Note that since $G$ is a multi-graph, it may have multiple edges between a given pair of vertices $u,v\in V(G)$. As such, we say two spanning trees $T_i$ and $T_j$ of $G$ are \emph{edge-disjoint} if whenever they both contain a $uv$-edge, these edges are distinct in $G$.} 
Formally, we regard the tuple $\K_1 = (K_1;T_1,\ldots,T_d)$, where $K_1$ is the graph with a single vertex and no edge{, and the edge sets of} $T_1,\ldots,T_d$ are empty, as a $d$-tree decomposition.
We denote by $\mathcal{G}_d$ the set of all $d$-tree decompositions.
Note that if a multi-graph $G$ admits a $d$-tree decomposition then it is necessarily loop-free and contains at most $d$ copies of any edge. 

\subsection{Realisations}
\label{Pre:Realisations}
Let $G=(V,E)$ be a finite  simple graph. A {\em placement} of   $G$ in $\mathbb{R}^d$ is an injective map $p:V\to \mathbb{R}^d$. The pair $(G,p)$ is referred to as a {\em (bar-joint) framework} in $\mathbb{R}^d$.
For each edge $uv\in E$, the pair $\{p(u),p(v)\}$ is said to be {\em well-positioned} if there exists a unique $k\in\{1,\ldots,d\}$ such that 
\[\| p(u)-p(v) \|_{\infty} = |(p(u)-p(v))\cdot e_k|.\]
Here $(e_i)_{i=1}^d$ is the usual basis for $\mathbb{R}^d$ and $\|x\|_\infty = \max_{1\leq i\leq d}|x\cdot e_i|$ is the $\ell^\infty$ norm of a vector $x\in \mathbb{R}^d$. 
The framework $(G,p)$ is said to be {\em well-positioned} if $\{p(u),p(v)\}$ is well-positioned for every edge $uv\in E$.

Each well-positioned framework $(G,p)$ in $\mathbb{R}^d$ inherits an edge-labelling $\kappa_p:E\to \{1,\ldots,d\}$, referred to as the {\em framework colouring}, where for each edge $uv\in E$,
\[\kappa_p(uv)=\{k\} \,\mbox{ if and only if } \,\,\|p(u)-p(v)\|_\infty = |(p(u)-p(v))\cdot e_k|.\]
The set of edges of colour $k$, $\kappa_p^{-1}(k)$, induces the  subgraph $G[\kappa_p^{-1}(k)]$ which is referred to as an induced {\em monochrome subgraph} of $G$. 
A {\em realisation} for a $d$-tree decomposition $\G=(G;T_1,\ldots,T_d)$, where $G$ is a simple graph, is a framework $(G,p)$ in $\mathbb{R}^d$ with the property that $T_1,\ldots,T_d$ are the induced monochrome subgraphs of $G$.

\begin{example}
Figure \ref{fig:W5Placements} illustrates three placements of the wheel graph $W_5$ in the plane together with the induced framework colourings. Note that $W_5$ admits exactly two distinct 2-tree decompositions (up to graph isomorphism) and these are realised by the placements in the left and centre of Figure \ref{fig:W5Placements}. These two placements are well-positioned as each edge affinely spans a line of slope $m$ with $|m|\not=1$. The rightmost placement is not well-positioned as the edges incident {to} $v_0$ affinely span lines of slope $\pm1$.
\end{example}

\begin{figure}
\centering
\begin{tikzpicture}[scale = 0.9]
	\begin{scope}[shift = {(-2,0)}]
	  \node[vertex, label = below left :$v_1$]  (v1) at (-1  ,-1) {};
	  \node[vertex, label = below right :$v_2$] (v2) at (1  ,-1) {};
	  
		 \node[vertex, label = left :$v_3$] 	(v3) at (-1  ,1.2) {};
	  \node[vertex, label =  right :$v_4$] 	(v4) at (1  ,0.6) {};
		
		 \node[vertex, label = below  :$v_0$] 	(v0) at (0.1, -0.1) {};
		
		\draw (v1) to (v2); \draw (v3) to (v4);
		\draw (v1) to (v0); \draw (v0) to (v4);
	  \draw[dashed] (v1) to (v3); \draw[dashed] (v2) to (v4);
		\draw[dashed] (v2) to (v0); \draw[dashed] (v0) to (v3);
		  
	\end{scope}
	\begin{scope}[shift = {(2,0)}]
	 \node[vertex, label = below left :$v_1$]  (v1) at (-1  ,-1) {};
	  \node[vertex, label = below right :$v_2$] (v2) at (0.2  ,-1) {};
	  
		 \node[vertex, label =  left :$v_3$] 	(v3) at (-1  ,1) {};
	  \node[vertex, label = below  :$v_4$] 	(v4) at (-0.7  ,0.2) {};
		
		 \node[vertex, label = right  :$v_0$] 	(v0) at (1, 0.5) {};
		
		\draw[dashed] (v0) to (v2); \draw[dashed] (v3) to (v4);
		\draw (v1) to (v0); \draw (v0) to (v4);
	  \draw[dashed] (v1) to (v3); \draw[dashed] (v2) to (v4);
		\draw (v2) to (v1); \draw (v0) to (v3);
	\end{scope}	
	\begin{scope}[shift = {(6,0)}]
	  \node[vertex, label = below left :$v_1$]  (v1) at (-1  ,-1) {};
	  \node[vertex, label = below right :$v_2$] (v2) at (1  ,-1) {};
	  
		 \node[vertex, label =  left :$v_3$] 	(v3) at (-1  ,1) {};
	  \node[vertex, label = right :$v_4$] 	(v4) at (1  ,1) {};
		
		 \node[vertex, label = below  :$v_0$] 	(v0) at (0, 0) {};
		
		\draw (v1) to (v2); \draw (v3) to (v4);
		
		\draw[lightgray] (v1) to (v0); \draw[lightgray] (v0) to (v4);
	  \draw[dashed] (v1) to (v3); \draw[dashed] (v2) to (v4);
		\draw[lightgray] (v2) to (v0); \draw[lightgray] (v0) to (v3);
		\end{scope}	
\end{tikzpicture}
\caption{Left and centre: Realisations of the two distinct $2$-tree decompositions for the wheel graph $W_5$. Right: A placement of $W_5$ which is not well-positioned.}
\label{fig:W5Placements}
\end{figure}
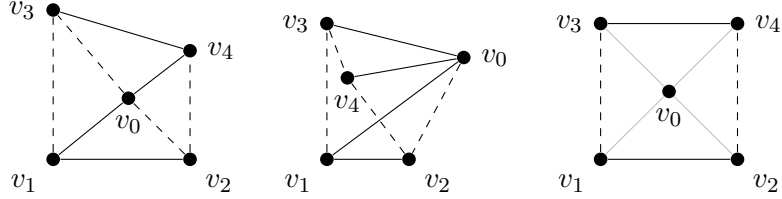

\subsection{Graph rigidity}
Motivation for considering realisation problems of this type comes from graph rigidity in normed spaces. Consider again a simple graph $G$. A framework $(G,p)$ is {\emph{(locally) rigid}} 
 with respect to a given norm on $\mathbb{R}^d$ if every edge-length preserving continuous motion of the vertices is obtained from an isometric motion of the space.
The notion of a well-positioned framework $(G,p)$, introduced above in the context of the $\ell^\infty$ norm, generalises to the condition that the {\emph{rigidity map}}, 
\[
f_G:(\bR^{d})^{V}\to \bR^{E},\,\,\,\,\,\,\, (x(v))_{v\in V}\mapsto (\|x(v)-x(w)\|)_{vw\in E},
\] 
is differentiable at $p=(p(v))_{v\in V}$. In the case of the $\ell^\infty$ norm on $\bR^d$, {\em infinitesimal rigidity} is the property that elements which lie in the kernel of the differential $df_G(p)$, also called the {\em infinitesimal flexes} of the framework, are translational, i.e.~ take the form $(a,\ldots,a)\in (\bR^{d})^{V}$ where $a\in \bR^d$.
It can be shown that, with respect to the $\ell^\infty$ norm,  local rigidity and infinitesimal rigidity are equivalent for well-positioned frameworks (see \cite{K_PolyhedralNorms}). In this case we simply refer to a framework as being {\em rigid}. A framework $(G,p)$ is {\em minimally rigid}  (or \emph{isostatic}) if it is rigid and removing any edge from $G$ results in a framework which is not rigid. 

Consider an $\ell^q$ norm on $\bR^d$ with $1\leq q\leq \infty$ and $q\not=2$. It is shown in \cite{KP_NonEuclideanNorms} that  a necessary condition for a well-positioned framework $(G,p)$ {on a simple graph $G$} to be isostatic with respect to $\ell^q$ is that $G$ is $(d,d)$-tight. It is also shown that in the case $d=2$, a converse statement holds:  every $(2,2)$-tight {simple} graph $G$ admits a well-positioned framework $(G,p)$ in the plane which is isostatic for the $\ell^q$ norm. It is conjectured that this converse statement extends to $d\geq 3$.
In the case of the $\ell^\infty$ norm, the spanning tree characterisation of $(d,d)$-tight graphs obtained by Nash-Williams \cite{N-W_SpanningTrees} and Tutte \cite{T_SpanningTrees}, together with the following result, support this conjecture and provide the link to the realisation problems considered in this article. 

\begin{thm}{\cite[Propositions 4.3 \& 4.4]{KP_NonEuclideanNorms}}
\label{thm:K_MinInfRigid_Simple}
Let $G$ be a simple graph and let $(G,p)$ be a well-positioned framework in $(\mathbb{R}^d, \|\cdot\|_\infty)$. The following statements are equivalent:
\begin{enumerate}
	\item $(G,p)$ is minimally  rigid;
	\item the monochrome subgraphs   induced by the framework colouring $\kappa_p$ are spanning trees in $G$.
	
\end{enumerate}
\end{thm}

To summarise, if $G$ is a simple graph which is $(d,d)$-tight then it admits a $d$-tree decomposition and so, by the above theorem, to show that $G$ admits a well-positioned minimally rigid framework in $(\mathbb{R}^d, \|\cdot\|_\infty)$ it is sufficient to prove that this $d$-tree decomposition can be realised in $\bR^d$.

\subsection{Inductive constructions}
Nash-Williams \cite{N-W_SpanningTrees} and Tutte \cite{T_SpanningTrees} independently  
characterised the multi-graphs which admit a $d$-tree decomposition as those which are $(d,d)$-tight:

\begin{thm}\label{thm:TNW_SpanningTreeSparsity}
A  multi-graph $G=(V,E)$ is expressible as an edge-disjoint union of $d$ spanning trees if and only if 
\begin{enumerate}
	\item $|E|=d|V|-d$, and \label{part:TNW_SpanningTreeSparsity_tight}
	\item $i(X)\leq d|X|-d$ for all $\emptyset\neq X \subseteq V$.\label{part:TNW_SpanningTreeSparsity_sparse}
\end{enumerate}
\end{thm}

During his work on the rigidity of body-bar frameworks, Tay \cite{Tay_Multigraphs1} found the following inductive construction of $(d,d)$-tight graphs: 

\begin{thm}\label{thm:Tay_ddtightInductiveConstruction}
A multi-graph $G=(V,E)$ is $(d,d)$-tight if and only if there exists a sequence of graphs
\[
K_1=G^{(1)} \rightarrow G^{(2)} \rightarrow \cdots \rightarrow G^{(n)}=G
\]
such that for all $2\leq i\leq n$, $G^{(i)}$ is obtained from $G^{(i-1)}$ by a $d$-dimensional $j$-extension, for some $0\leq j \leq d-1$. 
\end{thm}

A \emph{$d$-dimensional $j$-extension} of a graph $G=(V,E)$ forms a new graph $G'$ by first deleting some set of edges $F$ from $G$, with $|F|=j$, and then appending a new vertex $v$ to $G$, incident {to} $d+j$ new edges, such that $N_{G'}(v)\supseteq V_G(F)$. Note that in rigidity theory it is usually sufficient to only consider $j\leq d-1$, however in this article we will allow $j\geq d$ when $d=1$. The inverse of a $d$-dimensional $j$-extension is a \emph{$d$-dimensional $j$-reduction} which, given a graph $G'$, forms $G$ by deleting some vertex $v$ of degree $j+d$ and then adding $j$ edges between the vertices in $N_{G'}(v)$. Extensions and reductions of this type were first introduced by Henneberg \cite{H_HennebergMoves}, and so are also known as Henneberg moves and inverse Henneberg moves respectively. 

Tay later used the inductive construction in Theorem \ref{thm:Tay_ddtightInductiveConstruction} to obtain a new proof of Theorem \ref{thm:TNW_SpanningTreeSparsity}, see \cite{Tay_Multigraphs2}. Combining Theorems \ref{thm:TNW_SpanningTreeSparsity} and \ref{thm:Tay_ddtightInductiveConstruction} gives the following result, which is the starting point of this paper:

\begin{thm}\label{thm:TNWTay_SpanningTreeInductiveConstruction}
A multi-graph $G=(V,E)$ is the edge-disjoint union of $d$ spanning trees if and only if there exists a sequence of graphs
\[
K_1=G^{(1)} \rightarrow G^{(2)} \rightarrow \cdots \rightarrow G^{(n)}=G
\]
such that for all $2\leq i\leq n$, $G^{(i)}$ is obtained from $G^{(i-1)}$ by a $d$-dimensional $j$-extension, for some $0\leq j \leq d-1$. 
\end{thm} 

\section{Inductive construction for $d$-tree decompositions}\label{sec:AllD_Combinatorics}
Given a $d$-tree decomposition $\G=(G;T_1,\ldots,T_d)$, the construction in Theorem \ref{thm:TNWTay_SpanningTreeInductiveConstruction} can be adapted so that it not only constructs the multi-graph $G$, but simultaneously constructs each of the {monochrome} spanning trees $T_1, T_2, \ldots, T_d$ through a sequence of $1$-dimensional Henneberg moves. 
This extension is not difficult to see, 
however to the best of our knowledge this result has not appeared in the literature. For completeness, we prove {it} here. 

\begin{lem}\label{lem:HennebergMovesOnTrees}
Let $T=(V,E)$ be a simple  graph, and let $S=(U,F)$ be a subgraph of $T$ which is a tree. Suppose $T'=(V',E')$ is formed from $T$ by a $1$-dimensional $|F|$-extension which deletes the edges in $F$, and appends a vertex $v$ incident {to} $|F|+1$ new edges. Then $T$ is a tree if and only if $T'$ is a tree. 
\end{lem}

\begin{proof}
By the definition of a $d$-dimensional $j$-extension, $|E|=|E'|-1$ and 
$|V|=|V'|-1$. It follows that $|E'| = |V'| -1$ if and only if $|E| = |V| -1$.
So to prove that $T$ is a tree if and only if $T'$ is a tree, it remains to note that $T$ is connected if and only if $T'$ is connected.
\end{proof}

The following elementary fact, which is a consequence of Theorem \ref{thm:TNW_SpanningTreeSparsity} and the definition of $\mathcal{G}_d$, shall also be useful:

\begin{lem}\label{lem:MinDegree}
Let $G$ be a multi-graph, with $|V(G)|\geq 2$, which admits a $d$-tree decomposition for some $d\geq 1$. Then $d\leq\delta(G)\leq 2d-1$.
\end{lem}

This lemma implies we only need to consider $d$-dimensional $j$-extensions of $d$-tree decompositions where $0 \leq j \leq d-1$.

\begin{defn}
A $d$-tree decomposition $\G'=(G';T_1',\ldots,T_d')$ is said to be obtained from a $d$-tree decomposition $\G=(G;T_1,\ldots,T_d)$ by a {\em $d$-tree $j$-extension}, for some $0\leq j \leq d-1$,
if 
\begin{enumerate}
	\item $G'$ is obtained from $G$ by a $d$-dimensional $j$-extension, and,
	\item for each $1\leq i \leq d$, $T_i'$ is obtained from $T_i$ by a $1$-dimensional $k_i$-extension, for some $0\leq k_i \leq d-1$. 
\end{enumerate}
{Note that $\sum_{i=1}^{d} k_i =j$.}
\end{defn}

\begin{prop}\label{prop:InductiveConstruction_UnionAndTrees}
Let $\G'\in \G_d$ be a $d$-tree decomposition with $\G'\not=\K_1$. Then there exists a $d$-tree decomposition $\G\in \G_d$ such that 
	$\G'$ is a $d$-tree $j$-extension of $\G$. 
\end{prop}

\begin{proof}
Let $\G'=(G';T_1',\ldots,T_d')$ where $G'=(V',E')$ and $T_i'=(V',E_i')$ for each $i=1,\ldots,d$.
By Lemma \ref{lem:MinDegree}, there exists $v\in V'$ such that $d\leq d_{G'}(v) \leq 2d-1$.
For each tree $T_i'$, perform a Henneberg reduction at $v$ which forms the graph $T_i=(V'-v, E_i)$ by deleting $v$ and adding a set of edges $F_i$ between the vertices in $N_{T_i'}(v)$ such that $(N_{T_i'}(v),F_i)$ is a tree. 

Since $T_1',T_2',\ldots, T_d'$ are spanning trees of $G'$,  we know $d_{T_i'}(v)\geq 1$ for all $1\leq i\leq d$. By Lemma \ref{lem:MinDegree}, we have that
\[
\sum_{i=1}^{d} d_{T_i'}(v)=d_{G'}(v)\leq 2d-1.
\]
Hence $1 \leq d_{T_i'}(v)\leq d$ for all $1\leq i\leq d$.
 Since $(N_{T_i'}(v),F_i)$ is a tree, we must have $|F_i|=|N_{T_i'}(v)|-1=d_{T_i'}(v)-1$. In other words, the inverse Henneberg move which forms $T_i$ from $T_i'$ is a $1$-dimensional $k_i$-reduction, with $k_i=d_{T_i'}(v)-1$. Which implies $0\leq k_i \leq d-1$. 

Let $G=(V'-v,E_1\cup\cdots\cup E_d)$.
By Lemma \ref{lem:HennebergMovesOnTrees}, $T_1,\ldots,T_d$ are trees and so $\G=(G;T_1,\ldots,T_d)$ is a $d$-tree decomposition.
Finally, the move which forms $G$ from $G'$ deletes the vertex $v$, and then adds $\sum_{i=1}^d (d_{T_i'}(v)-1)=d_{G'}(v)-d$ edges between the vertices in $N_{G'}(v)$. This is, by definition, a $d$-dimensional $j$-reduction, where $j= d_{G'}(v)-d$, and since $d\leq d_{G'}(v)\leq 2d-1$, we have $0\leq j\leq d-1$. 
\end{proof}

\begin{cor}\label{cor:InductiveConstruction_TreesAndGraph}
Let $\G=(G;T_1,\ldots,T_d)$ be a $d$-tree decomposition.
Then,
 there exists a sequence of $d$-tree decompositions
\[
\K_1= \G^{(1)} \rightarrow \G^{(2)} \rightarrow \cdots \rightarrow \G^{(n)}=\G
\]
such that for all $2\leq i\leq n$, $\G^{(i)}$ is obtained from $\G^{(i-1)}$ by a $d$-tree $j$-extension, for some $0\leq j \leq d-1$.
\end{cor}

\section{Realisations for 2-tree decompositions}\label{sec:2D_Geometry} 
{In this section, we} apply our inductive construction from Corollary \ref{cor:InductiveConstruction_TreesAndGraph} to show that every $2$-tree decomposition has a plane realisation. 
To do this, we first extend the definitions of a well-positioned framework $(G,p)$ in $\bR^d$ and an induced framework colouring $\kappa_p$, which were given in Section \ref{Pre:Realisations}, to accommodate multi-graphs.  

Let $G=(V,E)$ be a multi-graph with no loops and 
let $p:V\to \mathbb{R}^d$ be an injective map. As before we refer to the pair $(G,p)$ as a {\emph{framework}}
in $\mathbb{R}^d$.
Suppose a pair of vertices $u,v\in V$ are joined by exactly $t$  edges, where $1\leq t\leq d$. 
The pair $\{p(u),p(v)\}$ is said to be {\em well-positioned} if  there exist exactly $t$ distinct elements  $j_1,\ldots,j_t\in\{1,\ldots,d\}$ such that  
\[
\| p(u)-p(v) \|_{\infty} = |(p(u)-p(v))\cdot e_{j_k}|
\]
for $k=1,2,\ldots, t$. We refer to $j_1,\ldots,j_t$ as the {\em framework colours} for the pair $\{p(u),p(v)\}$.
The framework $(G,p)$, on the multi-graph $G$, is said to be {\em well-positioned} if, for every edge $uv\in E(G)$, the pair $\{p(u),p(v)\}$ is well-positioned. 

When a framework $(G,p)$ is well-positioned, the multi-graph $G$ inherits an edge-labelling $\kappa_p : E{(G)} \rightarrow \{1,2,\ldots,d\}$ whereby each of the $t$ edges connecting a pair of vertices $u$ and $v$ is assigned one of the distinct framework colours $j_1,\ldots,j_t$ for the pair $\{p(u),p(v)\}$. The edge-labelling $\kappa_p$   is referred to as a {\em framework colouring} for $(G,p)$. Note that this framework colouring is unique up to permutation of colours between parallel edges. Such permutations will not create any problems in what follows. The subgraph $G[\kappa_p^{-1}(k)]$ spanned by edges with framework colour $k$ is again referred to as a {\em monochrome subgraph} of $G$.

{Given a $d$-tree decomposition $\G=(G;T_1,\ldots,T_d)$ of a multi-graph $G$, a {\em realisation} of $\G$} is a framework $(G,p)$ in $\mathbb{R}^d$ with the property that $T_1,\ldots,T_d$ are the induced monochrome subgraphs of $G$, for some choice of framework colouring $\kappa_p$. 
Note that by relabelling $T_1,\ldots,T_d$ we may assume, without loss of generality, that for $i=1,\ldots,d$, edges $uv$ in $T_i$ satisfy, 
\[\|p(u)-p(v)\|_{\infty} = |(p(u)-p(v))\cdot e_i|.\]

\begin{prop}\label{prop:InductiveConstruction_Realisation0}
Let $\G=(G;T_1,T_2)$ be a $2$-tree decomposition and suppose $\G'=(G';T_1',T_2')$ is a $2$-tree decomposition which is obtained by applying a $2$-tree $0$-extension to $\G$.

If $\G$ has a realisation $p$ in the plane then $\G'$ has a realisation $p'$ in the plane with the property that $p'(w)=p(w)$ for all $w\in V(G)$.
\end{prop}

\begin{proof}
Suppose the $2$-dimensional $0$-reduction which forms $G$ from $G'$ deletes the vertex $v$. 
	By the definition of a $2$-dimensional $0$-reduction, $d_{G'}(v)=2$, and $G$ is formed from $G'$ by deleting $v$ and the two edges incident {to} $v$. Since $T_1'$ and $T_2'$ are both spanning trees of $G'$, this implies that $v$ is a leaf node of both $T_1'$ and $T_2'$. Hence $T_i=T_i'-v$ for $i\in\{ 1,2 \}$.
	Let $N_{G'}(v)=\{x,y\}$ where $x$ and $y$ may or may not be distinct. Without loss of generality, suppose $vx\in E(T_1')$ and $vy\in E(T_2')$. Since $G=G'-v$, we know that $(G'-v,p)$ is well-positioned. So let $p'(w)=p(w)$ for all $w\in V(G)$. For $(G',p')$ to be well-positioned, it only remains to find a position for $p'(v)$ such that the pairs $\{p'(v),p'(x)\}$ and $\{p'(v),p'(y)\}$ are well-positioned. 
		
If  $x\neq y$ then $p'(x)\neq p'(y)$. If we place $v$ within a sufficiently small distance of the intersection of the lines  $p'(x)+\lambda (1,0)$ and  $p'(y)+\mu (0,1)$, then $\{p'(v),p'(x)\}$ and $\{p'(v),p'(y)\}$ will have framework colours $1$ and $2$ respectively. 
	
	Now suppose $x=y$. Then $v$ is met by a double-edge; one edge in $T_1'$ and the other in $T_2'$. To satisfy the constraints induced by this colouring, we must place $p'(v)$ such that
\[
\|p'(v)-p'(x)\|_{\infty} = |v_1-x_1| = |v_2-x_2|.
\]	
If we place $p'(v)$ at any point of the lines $p'(x)+\lambda(1,1)$ or $p'(x)+\lambda(1,-1)$, then we will satisfy these constraints, and, since $vx$ is a double-edge, the pair $\{p'(v),p'(x)\}$ will be well-positioned. 

In both cases we can choose  $p'(v)$ so that it is not coincident with another vertex of $G'$. Moreover, $(G',p')$ is well-positioned, and the induced monochrome subgraphs  are $T_1'$ and $T_2'$.
\end{proof}

To complete the inductive construction we  require a geometric method of realising  $2$-tree $1$-extensions which accommodates parallel edges. 

\begin{prop}\label{prop:InductiveConstruction_Realisation1}
Let $\G=(G;T_1,T_2)$ be a $2$-tree decomposition and suppose $\G'=(G';T_1',T_2')$ is a $2$-tree decomposition which is obtained by applying a  $2$-tree $1$-extension to $\G$.

If $\G$, and every $2$-tree decomposition with fewer vertices than $\G$, has a realisation $p$ in the plane then $\G'$ has a realisation $p'$ in the plane.
\end{prop}

\begin{proof}
 Suppose the $2$-dimensional $1$-reduction which forms $G$ from $G'$ deletes the vertex $v$. 	
	By the definition of a $2$-dimensional $1$-reduction, $d_{G'}(v)=3$, and $G$ is formed from $G'$ by deleting $v$ and all three edges incident {to} $v$, before adding a single edge between the vertices in $N_{G'}(v)$. Since $T_1'$ and $T_2'$ are both spanning trees of $G'$, this implies that $v$ is a leaf node of one of these trees, and is incident {to} two edges {of} the other tree. Without loss of generality assume $d_{T_1'}(v)=1$ and $d_{T_2'}(v)=2$. Let $N_{T_1'}(v)=\{x\}$	and $N_{T_2'}(v)=\{y,z\}$, where {$y\not\in\{x, z\}$,} but potentially $x=z$. Then,  $T_1=T_1'-v$. In order for $T_2$ to be a tree, it must be formed by deleting $v$ from $T_2'$, and then adding an edge $e$ between $y$ and $z$. In other words, $T_2=T_2'-v+e$.
Let $(G,p)$ be a realisation for $\G$ in the plane.	 {We shall use $p$ to construct a realisation $p'$ of $G'$.}
		
		For all pairs of points $(a,b),(c,d)\in \mathbb{R}^2$, write $(a,b)\leq (c,d)$ if $a\leq c$ and $b\leq d$ and write $(a,b)<(c,d)$ if, in addition, $(a,b)\not=(c,d)$.
Note that reflecting the set of points $\{p(w):w\in V(G)\}$  through either of the coordinate axes in $\bR^2$ generates another realisation of $\G$ in which the  framework colours for all pairs $\{p(u),p(w)\}$ are preserved. Thus we may assume, without loss of generality, that $p(z)<p(y)$.

\begin{case}Suppose  $x\neq z$.\\
If $yz\not\in E(T_1)$, then let $p'(w)=p(w)$ for all $w\in V(G)$. Since $yz\in E(T_2)$, any placement of $p'(v)$ on the line $p'(y)+\mu(p'(z)-p'(y))$, $\mu\in\bR$, which is distinct from $p'(y)$ and $p'(z)$, will ensure the pairs $\{p'(v),p'(y)\}$ and $\{p'(v),p'(z)\}$ are well-positioned with framework colour $2$. Similarly, for all $-1< a < 1$, any placement of $p'(v)$ on the line $p'(x)+\lambda(1,a)$, $\lambda\in\bR$, which is distinct from $p'(x)$, will ensure $\{p'(v),p'(x)\}$ is well-positioned with framework colour $1$. Hence, if we place $p'(v)$ at the point of intersection of the two lines $p'(x)+\lambda (1,a)$ and $p'(y)+\mu(p'(z)-p'(y))$, and $a$ is chosen so that this intersection is not coincident with any other vertex of $G'$, then $(G',p')$ is a realisation of $\G'$. See Figure \ref{subfig:MainProof_FirstSetOfCases_1}.

If $yz\in E(T_1)$, then proceed as follows: 
Let $o$ denote the intersection of the lines $p(x)+\lambda(1,0)$ and $p(y)+\mu(1,1)$ and choose $\epsilon>0$ and $\delta>0$. For a pair of vertices $u,w\in V(G)$, write $u\sim w$ if either $u=w$, or, $w$ is joined to $u$ by a sequence of parallel edges in $G-e$.

If $o\leq p(z)$ then, informally, $p'(v)$ may be placed to the right of $o$ and $p(z)$ translated upwards. See Figure \ref{subfig:MainProof_FirstSetOfCases_2}. Any vertex in $G$ which is joined to $z$ by a sequence of parallel edges in $G-e$  must also be translated upwards. Formally, define $p'$ by setting $p'(v) = o + \delta(1,0)$ and, for each $w\in V(G)$,
\[
p'(w)=
\begin{cases}
	p(w)+\epsilon(0,1)	& \text{if } w\sim z, 
	\text{ and, }\\
	p(w)							& \text{otherwise.}
\end{cases}
\]

If $p(z)< o<p(x)$ then, informally, place $p'(v)$ below $p(z)$ and then translate $p(z)$ upwards. See Figure \ref{subfig:MainProof_FirstSetOfCases_3}. Any vertex in $G$ which is joined to $z$ by a sequence of parallel edges in  $G-e$  must also be translated upwards. Formally, define $p'$ by setting $p'(v) = p(z) + \delta(0,-1)$ and, for each $w\in V(G)$,
\[
p'(w)=
\begin{cases}
	p(w)+\epsilon(0,1)	& \text{if }  w\sim z, 
	\text{ and, }\\
	p(w)							& \text{otherwise.}
\end{cases}
\]

If $p(z)<o\leq p(y)$ and $p(x)< o$ then, informally, place $p'(v)$ above $p(y)$ and then translate $p(y)$ downwards. See Figure \ref{subfig:MainProof_SecondSetOfCases_1}.
Any vertex in $G$ which is joined to $y$ by a sequence of parallel edges in $G-e$  must also be translated downwards.
Formally, define $p'$ by setting $p'(v) = p(y) + \delta(0,1)$ and, for each $w\in V(G)$, 
\[
p'(w)=
\begin{cases}
	p(w)+\epsilon(0,-1)	& \text{if } w\sim y, 
	\text{ and, }\\
	p(w)							& \text{otherwise.}
\end{cases}
\]

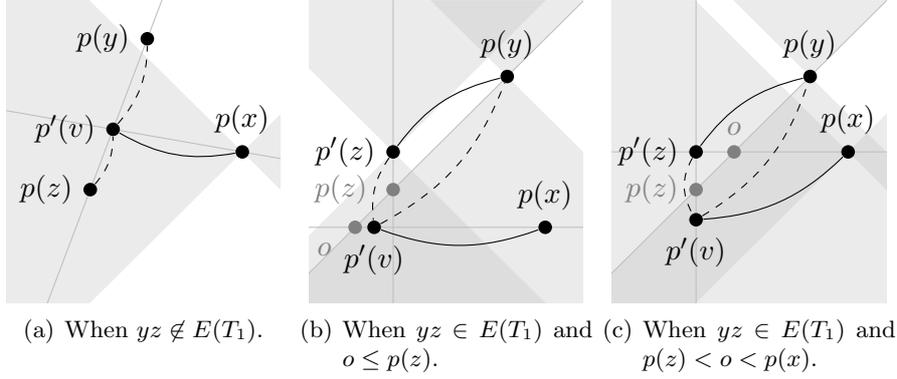
\begin{figure}
\centering
\subfigure[When $yz\not\in E(T_1)$.]
  {\label{subfig:MainProof_FirstSetOfCases_1}
	\begin{tikzpicture}[scale = .5]
	  \clip (-2.2,-3) rectangle (5,5); 
	  \node[vertex, label = left: $p(z)$] (z) at (0,0){}; 
	  \node[vertex, label = left: $p(y)$] (y) at (3-1.5,3+1){}; 
	  \node[vertex, label = $p(x)$] (x) at (4,1){}; 
	  \node (x') at ([shift=({170:-.2})]x) {}; 
	  \node[vertex, label = left: $p'(v)$] (v') at (intersection of y--z and x--x') {}; 
	  \draw[dashed, bend right = 20] (z) to (v') (v') to (y); 
	  \draw[bend right = 20] (v') to (x);
	  \begin{scope}[on background layer]	  
	    \clip (-2.2,-3) rectangle (5,5); 
	    \draw[extended line, lightgray] (z) to (y); 
	    \draw[extended line, lightgray] (v') to (x); 
	    \foreach \i in {0,1,2,3}{\node (x\i) at ([shift=({45 +\i*90:10})]x) {};} 
	    \path[fill=gray!50, opacity =.3] (x.center) -- (x1.center) -- (x2.center) -- (x.center);
	    \path[fill=gray!50, opacity =.3] (x.center) -- (x3.center) -- (x0.center) -- (x.center);
	  \end{scope}
	\end{tikzpicture}
  }
\subfigure[When $yz\in E(T_1)$ and $o\leq p(z)$.]
  {\label{subfig:MainProof_FirstSetOfCases_2}
	\begin{tikzpicture}[scale = .5]
	  \clip (-2.2,-3) rectangle (5,5); 
	  \node[vertex, color= gray, label = left: \color{gray} $p(z)\phantom{.}$] (z) at (0,0){}; 
	  \node[vertex, label = above: $p(y)$] (y) at (3,3){}; 
	  \node[vertex, label = $p(x)$] (x)  at ( 4,-1){}; 
	  \node(x') at ([shift=({0:-.2})]x) {}; 
	  \node[vertex, color = gray, label = below left: \color{gray} $o\phantom{.}$] (o) at (intersection of y--z and x--x') {}; 
	  \node[vertex, label = left: $p'(z)$] (z') at ([shift=({90:1})]z){}; 
	  \node[vertex, label = below: $p'(v)$] (v') at ([shift=({0:.5})]o){}; 
	  \draw[dashed, bend right = 20] (z') to (v') (v') to (y); 
	  \draw[bend right = 20] (v') to (x) (y) to (z');
	  \begin{scope}[on background layer]	  
	    \clip (-2.2,-3) rectangle (5,5); 
		  \draw[extended line, lightgray] (z) to (y); 
		  \draw[extended line, lightgray] (o) to (x); 
		  \draw[extended line, lightgray] (z) to (z'); 
	    \foreach \w in {y,z'} 
	      {
	      \foreach \i in {0,1,2,3}{\node (\w\i) at ([shift=({45 +\i*90:10})]\w) {};} 
		    \path[fill=gray!50, opacity =.3] (\w.center) -- (\w1.center) -- (\w0.center) -- (\w.center);
		    \path[fill=gray!50, opacity =.3] (\w.center) -- (\w3.center) -- (\w2.center) -- (\w.center); 
		    }  
	  \end{scope}
	\end{tikzpicture}
	  }
\subfigure[When $yz\in E(T_1)$ and $p(z)<o<p(x)$.]
  {\label{subfig:MainProof_FirstSetOfCases_3}
	\begin{tikzpicture}[scale = .5]
	  \clip (-2.2,-3) rectangle (5,5); 
	  \node[vertex, color= gray, label = left: \color{gray} $p(z)$] (z) at (0,0){}; 
	  \node[vertex, label = above: $p(y)$] (y) at (3,3){}; 
	  \node[vertex, label = $p(x)$] (x)  at ( 4,1){}; 
	  \node(x') at ([shift=({0:-.2})]x) {}; 
	  \node[vertex, color = gray, label = above: \color{gray} $o$] (o) at (intersection of y--z and x--x') {}; 
	  \node[vertex, label = left: $p'(z)$] (z') at ([shift=({90:1})]z){}; 
	  \node[vertex, label = below: $p'(v)$] (v') at ([shift=({-90:.8})]z){}; 
	  \draw[dashed, bend right = 20] (v') to (y); 
	  \draw[dashed, bend right = 30] (z') to (v');
	  \draw[bend right = 20] (v') to (x) (y) to (z');
	  \begin{scope}[on background layer]	  
	    \clip (-2.2,-3) rectangle (5,5); 
		  \draw[extended line, lightgray] (z) to (y); 
		  \draw[extended line, lightgray] (o) to (x); 
		  \draw[extended line, lightgray] (z) to (z'); 
	    \foreach \w in {x,y,z'} 
	      {\foreach \i in {0,1,2,3}{\node (\w\i) at ([shift=({45 +\i*90:10})]\w) {};}} 
	    \path[fill=gray!50, opacity =.3] (y.center) -- (y1.center) -- (y0.center) -- (y.center);
	    \path[fill=gray!50, opacity =.3] (y.center) -- (y3.center) -- (y2.center) -- (y.center);   
	    \path[fill=gray!50, opacity =.3] (x.center) -- (x1.center) -- (x2.center) -- (x.center);
	    \path[fill=gray!50, opacity =.3] (x.center) -- (x3.center) -- (x0.center) -- (x.center);  
	  \end{scope}
	\end{tikzpicture}
  }
\caption{Some choices for $p'(v)$ in $(G',p')$ when $x,y$ and $z$ are distinct. Features of $(G',p')$ are shown in black, with solid and dashed black lines denoting edges in $T_1'$ and $T_2'$ respectively. Unless otherwise indicated, $p'(w) = p(w)$ for all $w\in V(G)$.
For each vertex $w\in\{x,y,z\}$ we highlight in grey either a shaded region or a line, such that any placement $P$ of $v$ within this region ensures the pair $(P,p'(w))$ satisfies the geometric constraints implied by the graph colouring of $G'$. In each case, our chosen placement, $p'(v)$, lies within the intersection of these three regions.
Note that in examples \subref{subfig:MainProof_FirstSetOfCases_2} and \subref{subfig:MainProof_FirstSetOfCases_3}, both $T_1$ and $T_2$ contain a $yz$-edge, so to ensure the pair $(p'(y),p'(z))$ is well-positioned after $yz\in E(T_2)$ is deleted, we shift $z$ upwards.
}
\end{figure}

If $p(z)<p(y)<o$ and $p(x)< o$ then, informally, place $p'(v)$ above $o$ and then translate $p(y)$ downwards. See Figure \ref{subfig:MainProof_SecondSetOfCases_2}.
Any vertex in $G$ which is joined to $y$ by a sequence of parallel edges in $G-e$  must also be translated downwards.
Formally, define $p'$ by setting $p'(v) = o + \delta(0,1)$ and, for each $w\in V(G)$, 
\[
p'(w)=
\begin{cases}
	p(w)+\epsilon(0,-1)	& \text{if } w\sim y, 
	\text{ and, }\\
	p(w)							& \text{otherwise.}
\end{cases}
\]

If $p(z)<o$ and $p(x)=o$ then $p(x)$, $p(y)$ and $p(z)$ are collinear.
First suppose $p(z)<p(y)<p(x)$.  Place $p'(v)$ at the intersection of the horizontal line through $p(x)$ and the vertical line through $p(y)$. Then translate $p(y)$ downwards. See Figure \ref{subfig:MainProof_SecondSetOfCases_3}. Any vertex in $G$ which is joined to $y$ by a sequence of parallel edges in $G-e$  must also be translated downwards.
Thus,  for each $w\in V(G)$ we define, 
\[
p'(w)=
\begin{cases}
	p(w)+\epsilon(0,-1)	& \text{if } w\sim y, 
	\text{ and, }\\
	p(w)							& \text{otherwise.}
\end{cases}
\]
Now suppose $p(z)<p(x)<p(y)$.  In this case a convenient placement for $v$ is not to hand. Instead, contract the parallel edges between $y$ and $z$ to form a new $2$-tree decomposition $\G^o=(G^o;T_1^o,T_2^o)$. Write  $V(G^o)=V(G)-\{y,z\}+\{w_0\}$ where $w_0$ is the vertex obtained by identifying $y$ and $z$.
Since $\G^o$ has fewer vertices than $\G$, there exists a realisation $p^o$ for $\G^o$ in the plane. From this realisation, we can construct a new realisation $p$ for $\G$ with the property that $p(z)<p(y)<p(x)$. Formally, set $p(z)=p^o(w_0)$, $p(y)=p^o(w_0)+\epsilon'(1,1)$ for some sufficiently small $\epsilon'>0$, and $p(w)=p^o(w)$ for all remaining $w\in V(G)$. Now   proceed as above.

\begin{figure}
\centering

\subfigure[When $p(z)< o < p(y)$ and $p(x)<o$.]
	{\label{subfig:MainProof_SecondSetOfCases_1}
	\begin{tikzpicture}[scale = .5]
	  \clip (-2,-2) rectangle (5,6); 
	  \node[vertex, label = below: $p(z)$] (z) at (0,0){}; 
	  \node[vertex, color = gray, label = right: \color{gray} $p(y)$] (y) at (3,3){}; 
	  \node[vertex, label = below: $\phantom{m}p(x)$] (x)  at (-1.5,1){}; 
	  \node(x') at ([shift=({0:-.2})]x) {}; 
	  \node[vertex, color = gray, label = above: \color{gray} $o$] (o) at (intersection of y--z and x--x') {}; 
	  \node[vertex, label = right: $p'(y)$] (y') at ([shift=({-90:1})]y){}; 
	  \node[vertex, label = right: $p'(v)\phantom{mmn}$] (v') at ([shift=({90:1})]y){}; 
	  \draw[dashed, bend left = 20] (z) to (v');
	  \draw[dashed, bend right = 30] (v') to (y'); 
	  \draw[bend right = 20] (v') to (x) (z) to (y');
	  \begin{scope}[on background layer]	  
	    \clip (-2,-2) rectangle (5,6); 
		  \draw[extended line, lightgray] (z) to (y); 
		  \draw[extended line, lightgray] (o) to (x); 
		  \draw[extended line, lightgray] (y) to (y'); 
	    \foreach \w in {x,z} 
	      {\foreach \i in {0,1,2,3}{\node (\w\i) at ([shift=({45 +\i*90:10})]\w) {};}} 
	    \path[fill=gray!50, opacity =.3] (z.center) -- (z1.center) -- (z0.center) -- (z.center);
	    \path[fill=gray!50, opacity =.3] (z.center) -- (z3.center) -- (z2.center) -- (z.center);   
	    \path[fill=gray!50, opacity =.3] (x.center) -- (x1.center) -- (x2.center) -- (x.center);
	    \path[fill=gray!50, opacity =.3] (x.center) -- (x3.center) -- (x0.center) -- (x.center);  
	  \end{scope}
	\end{tikzpicture}
	}
\subfigure[When $p(z) < p(y) < o$ and $p(x) < o$.]
	{\label{subfig:MainProof_SecondSetOfCases_2}  
	\begin{tikzpicture}[scale = .5]
	  \clip (-2,-2) rectangle (5,6); 
	  \node[vertex, label = left: $p(z)$] (z) at (0,0){}; 
	  \node[vertex, color = gray, label = right: \color{gray} $p(y)$] (y) at (3,3){}; 
	  \node[vertex, label = $p(x)$] (x)  at (1,4){}; 
	  \node(x') at ([shift=({0:-.2})]x) {}; 
	  \node[vertex, color = gray, label = right: \color{gray} $o$] (o) at (intersection of y--z and x--x') {}; 
	  \node[vertex, label = right: $p'(y)$] (y') at ([shift=({-90:1})]y){}; 
	  \node[vertex, label = above: $p'(v)$] (v') at ([shift=({90:.7})]o){}; 
	  \draw[dashed, bend left = 20] (z) to (v');
	  \draw[dashed, bend right = 10] (v') to (y'); 
	  \draw[bend right = 20] (v') to (x) (z) to (y');
	  \begin{scope}[on background layer]	  
	    \clip (-2,-2) rectangle (5,6); 
		  \draw[extended line, lightgray] (z) to (y); 
		  \draw[extended line, lightgray] (o) to (x); 
		  \draw[extended line, lightgray] (y) to (y'); 
		  	\draw[extended line, lightgray] (o) to (v'); 
	    \foreach \w in {y',z} 
	      {\foreach \i in {0,1,2,3}{\node (\w\i) at ([shift=({45 +\i*90:10})]\w) {};}} 
	    \foreach \w in {y',z}
	      {
	      \path[fill=gray!50, opacity =.3] (\w.center) -- (\w1.center) -- (\w0.center) -- (\w.center);
	      \path[fill=gray!50, opacity =.3] (\w.center) -- (\w3.center) -- (\w2.center) -- (\w.center);   
	      }
	\end{scope}
	\end{tikzpicture}
	}
\subfigure[When $p(z) < p(y) < p(x)$ and $p(x)= p(o)$.]
	{\label{subfig:MainProof_SecondSetOfCases_3}
	\begin{tikzpicture}[scale = .5]
	  \clip (-2,-2) rectangle (5,6); 
	  \node[vertex, label = left: $p(z)$] (z) at (0,0){}; 
	  \node[vertex, color = gray, label = right: \color{gray} $p(y)$] (y) at (3,3){}; 
	  \node[vertex, label = above: $p(x)\phantom{r}$] (x)  at (4.5,4.5){}; 
	  \node(x') at ([shift=({0:-.2})]x) {}; 
	  \node[vertex, label = right: $p'(y)$] (y') at ([shift=({0,-1})]y){}; 
	  \node[vertex, label = above left: $p'(v)$] (v') at (intersection of y--y' and x--x') {}; 
	  \draw[dashed, bend left = 20] (z) to (v');
	  \draw[dashed, bend right = 30] (v') to (y'); 
	  \draw[bend right = 30] (v') to (x);
	  \draw[bend right = 20] (z) to (y');
	  \begin{scope}[on background layer]	  
	    \clip (-2,-2) rectangle (5,6); 
		  \draw[extended line, lightgray] (z) to (y); 
		  \draw[extended line, lightgray] (y) to (y'); 
		  	\draw[extended line, lightgray] (x) to (x'); 
	    \foreach \w in {z} 
	      {\foreach \i in {0,1,2,3}{\node (\w\i) at ([shift=({45 +\i*90:10})]\w) {};}} 
	    \foreach \w in {z}
	      {
	      \path[fill=gray!50, opacity =.3] (\w.center) -- (\w1.center) -- (\w0.center) -- (\w.center);
	      \path[fill=gray!50, opacity =.3] (\w.center) -- (\w3.center) -- (\w2.center) -- (\w.center);   
	      }
		  \end{scope}
	\end{tikzpicture}
	}
	\caption{Further choices for $p'(v)$ in $(G',p')$ when $x, y$ and $z$ are distinct and $yz\in E(T_1)$. In each case, we place $p'(y)$ below $p(y)$ to ensure the pair $(p'(y),p'(z))$ is well-positioned after $yz\in E(T_2)$ is deleted.
	}\label{fig:MainProof_SecondSetOfCases}
\end{figure}
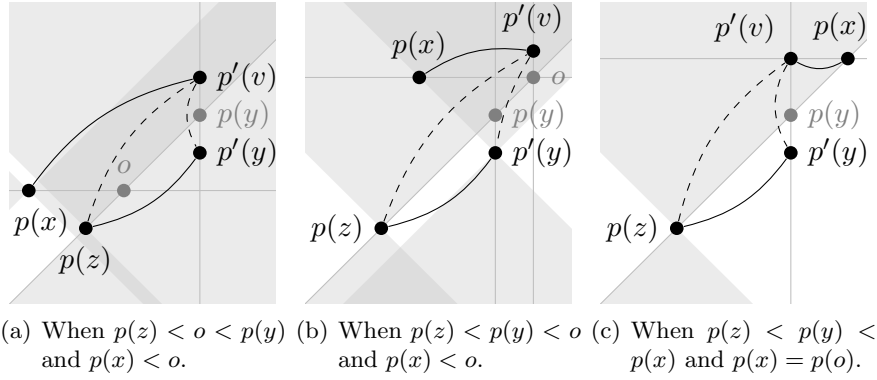
\end{case}

\begin{case} Suppose $x=z$. \\ 
In this case $v$ sends a double edge to $x$: one edge in $T_1'$ and the other in $T_2'$. If $xy\not\in E(T_1)$, then let $p'(w)=p(w)$ for all $w\in V(G)$. Place $p'(v)$ at the intersection of the lines $p'(x)+\lambda(1,1)$  and $p'(y)+\mu (a,1)$
where $-1<a<1$. Choose $a$ such that $p'(v)$ is not coincident with any other vertex of $(G',p')$. See Figure \ref{subfig:xyNotInT1}.

If $xy\in E(T_1)$ then, informally, place $p'(v)$ at $p(y)$ and then translate $p(y)$ downwards. Any vertex in $G$ which is joined to $y$ by a sequence of parallel edges in $G-e$  must also be translated downwards. Formally, define $p'$ by setting $p'(v)=p(y)$ and, for each $w\in V(G)$,
\[
p'(w)=
\begin{cases}
	p(w)+\epsilon(0,-1)	& \text{if } w\sim y, 
	\text{ and, }\\
	p(w)							& \text{otherwise,}
\end{cases}
\]
where $\epsilon>0$. See Figure \ref{subfig:xyInT1}

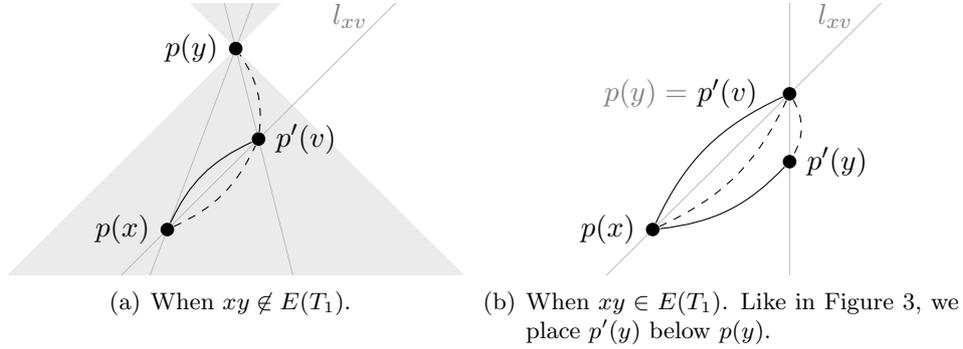
\begin{figure}[bht]
\centering
\subfigure[When $xy\not\in E(T_1)$.
]
  {\label{subfig:xyNotInT1}
	\begin{tikzpicture} [scale = .6]
		\node[label= \color{gray} $l_{xv}$] (l_xv) at (4,4) {}; 
	  \clip (-3.5,-1) rectangle (6.5,5); 
	  \node[vertex, label = left: $p(x)$] (x) at (0,0){}; 
	  \node[vertex, label = left: $p(y)$] (y) at (3-1.5,3+1){}; 

	  \node (x') at ([shift=({170:-.2})]x) {}; 
	  \node[vertex, label = right: $p'(v)$] (v') at (2,2) {}; 
	  \draw[dashed, bend right = 20] (x) to (v') (v') to (y); 
	  \draw[bend right = 20] (v') to (x);
	  \begin{scope}[on background layer]	  
	    \clip (-3.5,-1) rectangle (6.5,5); 
	    \draw[extended line, lightgray] (x) to (y); 
	    \draw[extended line, lightgray] (v') to (x); 
	    \draw[extended line, lightgray] (y) to (v'); 
	    \foreach \i in {0,1,2,3}{\node (y\i) at ([shift=({45 +\i*90:10})]y) {};} 
	    \path[fill=gray!50, opacity =.3] (y.center) -- (y1.center) -- (y0.center) -- (y.center);
	    \path[fill=gray!50, opacity =.3] (y.center) -- (y3.center) -- (y2.center) -- (y.center);
	  \end{scope}
	\end{tikzpicture}
  }  
\subfigure[When $xy\in E(T_1)$. Like in Figure \ref{fig:MainProof_SecondSetOfCases}, we place $p'(y)$ below $p(y)$.
]
  {\label{subfig:xyInT1}
	\begin{tikzpicture}[scale = .6]
		\node[label= \color{gray} $l_{xv}$] (l_xv) at (4,4) {}; 
	  \clip (-3.5,-1) rectangle (6.5,5); 
	  \node[vertex, label = left: $p(x)$] (x) at (0,0){}; 
	  \node[vertex, label = left: ${\color{gray}p(y)\eq }\ p'(v) \phantom{r}$] (v'y) at (3,3){}; 
	  \node (x') at ([shift=({170:-.2})]x) {}; 
	  \node[vertex, label = right: $p'(y)$] (y') at ([shift=({0,-1.5})]v'y){}; 
	  \draw[dashed, bend right = 20] (x) to (v'y);
	  \draw[dashed, bend left = 30] (v'y) to (y'); 
	  \draw[bend right = 20] (v'y) to (x);
	  \draw[bend left = 20] (y') to (x);	  
	  \begin{scope}[on background layer]	  
	    \clip (-3.5,-1) rectangle (6.5,5); 
	    \draw[extended line, lightgray] (x) to (v'y); 
	    \draw[extended line, lightgray] (v'y) to (y'); 
		  \end{scope}
	\end{tikzpicture}
  } 
\caption{Choices for $p'(v)$ in $(G',p')$ when $x=z$. Since both $T_1'$ and $T_2'$ contain an $xv$-edge, the line $l_{xv}$ through $p'(x)$ and $p'(v)$ has slope $1$. 
}
\end{figure}\label{fig:MainProof_x=z}
\end{case}

In each case, $\epsilon$ and $\delta$ can be chosen sufficiently small so that $p'(v)$ satisfies the required constraints and so that the points $\{p'(w):w\in V(G')\}$ are distinct. Thus $(G',p')$ is a realisation for $\G'$.
\end{proof}

We can now prove that every $2$-tree decomposition has a realisation in the plane.

\begin{thm}\label{thm:RealisationIn2Dimensions}
Let $\G=(G;T_1,T_2)$ be a $2$-tree decomposition. Then there exists a realisation for $\G$ in the plane.
\end{thm} 

\begin{proof}
Corollary \ref{cor:InductiveConstruction_TreesAndGraph} gives a sequence of $2$-tree decompositions and $2$-tree $0$ and $1$-extensions which construct $\G$ from the base element $\K_1$. Note that a realisation of $\K_1$ is obtained   by placing the vertex of $\K_1$ anywhere in the plane. By Propositions \ref{prop:InductiveConstruction_Realisation0} and \ref{prop:InductiveConstruction_Realisation1}, there exists a realisation for every $2$-tree decomposition in this sequence, in particular, such a realisation exists for $\G$.
\end{proof}

As a corollary we obtain an alternative proof of the following result from \cite{KP_NonEuclideanNorms}.

\begin{cor}\label{cor:RealisationIn2Dimensions}
Let $G$ be a $(2,2)$-tight simple graph. Then there exists a placement $p$ such that $(G,p)$ is well-positioned and minimally rigid in $(\mathbb{R}^2, \|\cdot \|_\infty)$.
\end{cor}

\proof
By Theorem  \ref{thm:TNW_SpanningTreeSparsity}, $G$ admits a $2$-tree decomposition $\G=(G;T_1,T_2)$.
By Theorem \ref{thm:RealisationIn2Dimensions}, this $2$-tree decomposition has a realisation $(G,p)$ in the plane.
By Theorem \ref{thm:K_MinInfRigid_Simple}, this realisation is  minimally rigid  in $(\mathbb{R}^2, \|\cdot \|_\infty)$.
\endproof

\section{Symmetric $2$-tree decompositions}
\label{sec:Symmetry} 

In this section we adapt the methods of the previous pages to show that every {\em symmetric} $2$-tree decomposition, with no fixed edges, can be realised as a symmetric framework in the plane which is minimally rigid with respect to $\ell^\infty$ distance constraints. We focus on frameworks with reflectional symmetry through a coordinate axis.  
Our motivation comes from recent work (\cite{KS_IncidentalSymm}) which characterises the class of symmetric graphs which admit a symmetric and minimally rigid realisation in the plane, in terms of certain symmetric $2$-tree decompositions. The main results of this section make the task of constructing  examples of symmetric minimally rigid frameworks  significantly easier. These results also suggest that further adaptations  may be possible in other contexts, such as gain graph constructions for symmetric frameworks.  

By a   \emph{$\bZ_2$-symmetric multi-graph} we will mean a pair $(G,\theta)$ consisting of a multi-graph $G$, with no loops, and a non-trivial group homomorphism $\theta: \bZ_2 \to \text{Aut}(G)$. 
Let $\bZ_2=\langle s\rangle$. To simplify notation, we denote $\theta(s)$ by $s_\theta$ and for each edge $e=v_1v_2$ we write $s_\theta(e)=s_\theta(v_1) s_\theta(v_2)$. The \emph{vertex orbit} of $v\in V(G)$ is the pair of vertices 
$\{v,s_\theta(v)\}$ and the \emph{edge orbit} of $e\in E(G)$ is $\{e,s_\theta(e)\}$. We say that a vertex $v$ (respectively an edge $e$) is \emph{fixed} if $v=s_\theta(v)$ (respectively $e=s_\theta(e)$). 

\begin{defn}
A {\em symmetric $2$-tree decomposition} is a tuple $\G=(G;T_1,T_2;\theta)$ such that,
\begin{enumerate}[(a)]
\item $(G;T_1,T_2)$ is a $2$-tree decomposition,
\item $(G,\theta)$ is a $\bZ_2$-symmetric multi-graph, and
\item $s_\theta(T_1)=T_1$ and $s_\theta(T_2)=T_2$.
\end{enumerate}
\end{defn}

Denote by $\G_2^{sym}$ the set of all symmetric $2$-tree decompositions with no fixed edges. We formally include in $\G_2^{sym}$ the tuple $\K_1=(K_1;T_1,T_2;\theta)$ where $K_1$ is the graph with a single vertex $v_0$ and no edges, $T_1$ and $T_2$ have empty edge set, and $\theta:\bZ_2\to \Aut(K_1)$ is the trivial group homomorphism with $s_\theta(v_0)=v_0$. This tuple will form the base element of a construction scheme for $\G_2^{sym}$.

\begin{lem} \label{lem:W_fixedEandV}
Let $\G=(G;T_1,T_2;\theta)\in\G_2^{sym}$ with $\G\not=\K_1$. Then $s_\theta$ fixes exactly one vertex which has even degree at least 4.
\end{lem}

\begin{proof}
The case where $G$ is a simple graph is proved in \cite[Lemma 3]{KS_IncidentalSymm} and this proof extends to the multi-graph case. 
\end{proof}

\subsection{Multi-graph construction scheme for $\G_2^{sym}$}

The following two graph moves were applied in the context of simple graphs, and with $d=2$, in \cite{KS_IncidentalSymm} (where they were referred to as $\bZ_2$-symmetric $1$-extensions and $\bZ_2$-symmetric $2$-extensions). Here we will first extend the moves to multi-graphs, allowing $d\in\{1,2\}$, and then introduce corresponding moves for symmetric $2$-tree decompositions.  

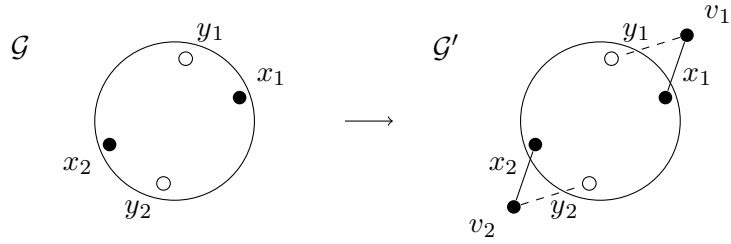
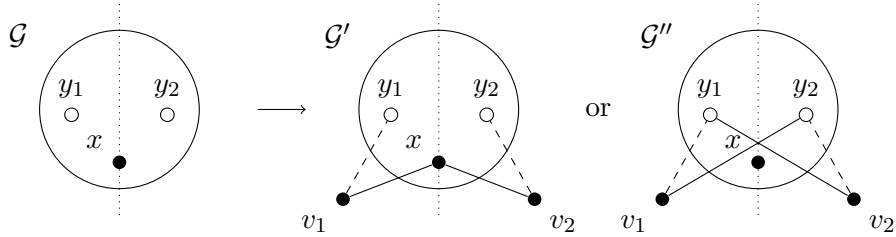
\begin{figure}
\centering
\subfigure[A symmetric $2$-tree decomposition $\G'$  formed from a symmetric $2$-tree decomposition $\G$ by a symmetric 2-tree 0-extension.] 
	{\label{subfig:0ext_halfturn}
	\begin{tikzpicture}[scale = 0.7]
	\foreach \xoffset/\name in {0/$\G$,8/$\G'$}
		{
		\begin{scope}[shift={(\xoffset,0)}]		
			\draw (0,0)	circle (1.5cm); 
			\draw (-2.9,1.4) node {\name};	
			\foreach \rotnum/\rotname in {0/$_1$, 180/$_2$}		
			{
					\node[vertex, label = \rotnum+20 : $x$\rotname] (x) at (20+\rotnum : 1.3cm) {};
					\node[emptyvertex, label = \rotnum+85 : $y$\rotname] (y) at (80+\rotnum : 1.2cm) {};
					\ifthenelse {\xoffset=8}{		
							\node[vertex, label = \rotnum +20 : $v$\rotname] (v) at (45+\rotnum : 2.3cm) {};
							\ifthenelse{\rotnum=2 \OR \rotnum=4}
							{	
								\draw[dashed] (v)--(x);
								\draw (v)--(y);
							}{
								\draw (v)--(x);
								\draw[dashed] (v)--(y);
							}
					}{}									
			}
			\ifthenelse{\xoffset = 0}{\draw[->]	(3.2,0) -- (4.1,0);}{}
		\end{scope}	
		}
	\end{tikzpicture}
	}
\subfigure[Two symmetric $2$-tree decompositions, $\G'$ and $\G''$, each formed by a  symmetric 2-tree 0-extension on the symmetric $2$-tree decomposition $\G$.]
	{\label{subfig:0ext_mirror}
	\begin{tikzpicture}[scale = 0.7]
	\foreach \xoffset/\name in {0/$\G$,6/$\G'$,12/$\G''$}
		{
		\begin{scope}[shift={(\xoffset,0)}]		
			\draw (0,0)	circle (1.5cm); 
			\draw[dotted] (0,-2) to (0,2); 
			\draw (-1.9,1.4) node {\name};	
			\node[vertex, label = 150 :$x$] 	(x) at (0  ,-1) {}; 
			\node[emptyvertex, label = above :$y_2$] 	 (y)  at ( 0.9,-0.1) {};
			\node[emptyvertex, label = above :$y_1$]  (y') at (-0.9,-0.1) {}; 
			\ifthenelse{\xoffset > 5}{
					\node[vertex, label = -90 + 40 :$v_2$]   (v) at (1.8  ,-1.7) {};
					\node[vertex, label = -90 + -40 :$v_1$] (v') at (-1.8  ,-1.7) {};
			}{}
			\ifthenelse {\xoffset=6}{			
					\draw (v)--(x)--(v');
					\draw[dashed] (v)--(y) (v')--(y');	
			}{}						
			\ifthenelse {\xoffset=12}{
					\draw[dashed] (v)--(y) (v')--(y');
					\draw (v)--(y') (v')--(y);
			}{}		
	    \ifthenelse{\xoffset = 0}{\draw[->]	(2.6,0) -- (3.5,0);}{}
			\ifthenelse{\xoffset = 6}{\node[label = center :\textnormal{or}] (or) at (3,0) {};}{}			
		\end{scope}	
		}
	\end{tikzpicture}
	}
\caption{Symmetric 2-tree 0-extensions under half-turn \subref{subfig:0ext_halfturn} and single mirror \subref{subfig:0ext_mirror} symmetry.}
\label{fig:0ext} 
\end{figure}

\begin{defn}
A $\bZ_2$-symmetric multi-graph $(G',\theta')$ is said to be obtained from a $\bZ_2$-symmetric multi-graph $(G,\theta)$ by a \emph{symmetric $d$-dimensional 0-extension}  if,
\begin{enumerate}[(a)]
\item $V(G')=V(G)\cup \{v,s_{\theta'}(v)\}$ where $v,s_{\theta'}(v)\notin V(G)$ and $v\not=s_{\theta'}(v)$,
\item $s_{\theta'}\vert_{V(G)}=s_\theta$,
\item $E(G')=E(G)+\{vv_i,s_{\theta'}(vv_i):i=1,\ldots, d\}$ for some $v_1,\ldots,v_d\in V(G)$ not necessarily distinct.
\end{enumerate}
\end{defn}
See Figure \ref{fig:0ext} for examples of such a move when $d=2$.

\begin{defn}
A $\bZ_2$-symmetric multi-graph $(G',\theta')$ is said to be obtained from a $\bZ_2$-symmetric multi-graph $(G,\theta)$ by a \emph{symmetric $d$-dimensional 1-extension}  if,
\begin{enumerate}[(a)]
\item $V(G')=V(G)\cup \{v,s_{\theta'}(v)\}$ where $v,s_{\theta'}(v)\notin V(G)$ and $v\not=s_{\theta'}(v)$,
\item $s_{\theta'}\vert_{V(G)}=s_\theta$,
\item there exist $d+1$ vertices $v_1,\ldots,v_{d+1}\in V(G)$, with $e=v_1v_2\in E(G)$ but which are otherwise not necessarily distinct, such that 
$E(G')=E(G)-\{e,s_{\theta}(e)\}+\{vv_i,s_{\theta'}(vv_i):i=1,\ldots,d+1\}$.
\end{enumerate}
\end{defn}
See Figure \ref{fig:1ext} for examples when $d=2$.

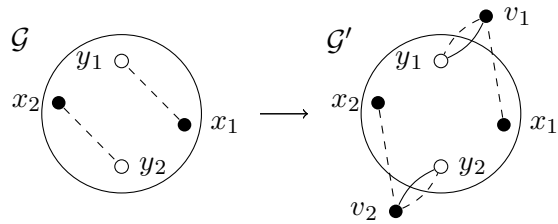
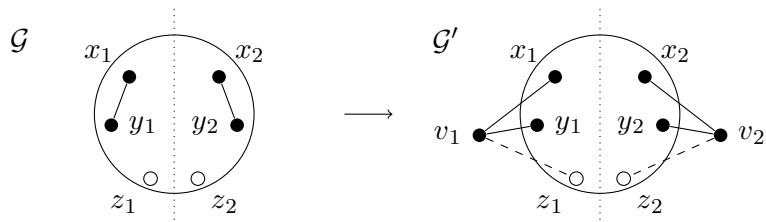
\begin{figure}
\centering
\subfigure[Half-turn symmetry] 
	{\label{subfig:1ext_halfturn}
	\begin{tikzpicture}[scale = 0.7]
	\foreach \xoffset/\name in {0/$\G$,6/$\G'$}
		{
		\begin{scope}[shift={(\xoffset,0)}]		
			\draw (0,0)	circle (1.5cm); 
			\draw (-1.9,1.4) node {\name};	
			\foreach \rotnum/\rotname in {0/$_1$, 180/$_2$}		
			{
				\node[vertex, label = \rotnum :\phantom{.}$x$\rotname] (x) at (-10+\rotnum : 1.2cm) {};
	  		\node[emptyvertex, label = \rotnum+180 :$y$\rotname] (y) at (90+\rotnum : 1cm) {};
				\ifthenelse {\xoffset=0}{
					\draw[->]	(2.6,0) -- (3.5,0);
					\draw[dashed] (x)--(y);
				}{
					\node[vertex, label = \rotnum :$v$\rotname] (v) at ([shift={(\rotnum+45:1.2)}]y) {}; 
					\draw[dashed] (x)--(v);
					\draw[bend left= 20] (v) to (y);
					\draw[bend right = 20, dashed] (v) to (y);
				}							
			}
		\end{scope}	
		}
	\end{tikzpicture}
	}
\subfigure[Single mirror symmetry]
  {\label{subfig:1ext_mirror}
	\begin{tikzpicture}[scale = 0.7]
	\foreach \xoffset/\name in {0/$\G$,8/$\G'$}
	{
		\begin{scope}[shift={(\xoffset,0)}]		
			\draw (0,0)	circle (1.5cm); 
			\draw[dotted] (0,-2) to (0,2); 
			\draw (-2.9,1.4) node {\name};	
			\foreach \reflnum/\reflname in {-1/$_2$, 1/$_1$} 
			{
				\node[vertex, label = 90+\reflnum*+50 :$x$\reflname] (x) at (90+50*\reflnum : 1.1cm) {};
				\node[vertex, label = 90+\reflnum*-90 :$y$\reflname] (y) at (90+100*\reflnum : 1.2cm) {};
			  \node[emptyvertex, label = 90+\reflnum*160 :$z$\reflname] (z) at (90+160*\reflnum : 1.3cm) {};
				\ifthenelse {\xoffset=0}{							
				  \draw (x) to (y);
				}{
					\node[vertex, label = 90+\reflnum*90 :$v$\reflname] (v) at (90+100*\reflnum : 2.3cm) {};					
						\draw (x) to (v) to (y);
						\draw[dashed] (v) to (z);
				}
			}
	    \ifthenelse{\xoffset = 0}{\draw[->]	(3.2,0) -- (4.1,0);}{}		
		\end{scope}	
	}
	\end{tikzpicture}
	}
\caption{Symmetric 2-tree 1-extensions. In each case, the symmetric 2-tree decomposition $\G'$ is obtained from the symmetric 2-tree decomposition $\G$ by a symmetric 2-tree 1-extension.}
\label{fig:1ext}
\end{figure}

We now adapt the above symmetric graph moves to incorporate symmetric $2$-tree decompositions. Again see Figures \ref{fig:0ext} and \ref{fig:1ext} for illustrations of these moves.

\begin{defn}
A symmetric $2$-tree decomposition $\G'=(G';T_1',T_2';\theta')$ is said to be obtained from a symmetric $2$-tree decomposition $\G=(G;T_1,T_2;\theta)$ by a {\em symmetric $2$-tree $j$-extension}, where $j\in \{0,1\}$,
if 
\begin{enumerate}
	\item $(G',\theta')$ is obtained from $(G,\theta)$ by a symmetric $2$-dimensional $j$-extension, and,
	\item for $i\in\{1,2\}$, $(T_i',\theta')$ is obtained from $(T_i,\theta)$ by a symmetric $1$-dimensional $k_i$-extension, for some $k_i\in\{0,1\}$ where $k_1+k_2=j$. 
\end{enumerate}
\end{defn}
Reconsider Figures \ref{fig:0ext} and \ref{fig:1ext}, this time noting the edge colourings.

We now prove the existence of a construction scheme for symmetric $2$-tree decompositions in $\G_2^{sym}$ which uses symmetric $2$-tree $0$  and $1$-extensions.

\begin{thm}\label{thm:W_GraphConstruction}
Let $\G'=(G';T_1',T_2';\theta')\in\G_2^{sym}$ with $\G'\not=\K_1$. Then there exists $\G=(G;T_1,T_2;\theta)\in\G_2^{sym}$ such that $\G'$ is obtained from $\G$ by either a symmetric 2-tree 0-extension or a symmetric 2-tree 1-extension.
\end{thm}

\begin{proof}
By Lemma \ref{lem:MinDegree}, $G'$ has a vertex $v$ of degree 2 or 3, and by Lemma \ref{lem:W_fixedEandV}, this vertex is not fixed. Hence the vertex orbit $\{v, s_{\theta'}(v)\}$ contains two distinct vertices with $d_{G'}(v) = d_{G'}(s_{\theta'}(v))$. Since $G'$ is the edge-disjoint union of spanning trees $T_1'$ and $T_2'$, $v$ is incident to at least one edge $e_i\in E(T_i')$ from each of these trees. Since $d_{G'}(v) \leq 3$, either $v$ is incident to one other edge, $e_3$, or $v$ is incident to no other edges.

\begin{claim}
If $v$ is incident to no other edges, then $\G'$ is formed from a symmetric $2$-tree decomposition $\G$ by a symmetric $2$-tree 0-extension.
\end{claim}

\begin{proof}
In this case $N_{G'}(v) =\{u_1,u_2\}$, where $u_1$ and $u_2$ need not be distinct. Since $d_{G'}(v) = 2$, and $s_{\theta'}(v)\not\in N_{G'}(v)$, we can remove both $v$ and $s_{\theta'}(v)$ from $G'$ by separate $2$-dimensional $0$-reductions to obtain a $\bZ_2$-symmetric multi-graph $(G,\theta)$ with $s_{\theta}=s_{\theta'}\vert_{V(G)}$. Let $T_i = T_i' - {\{v, s_{\theta'}(v)\}}$ and note that $s_\theta(T_i)=T_i$. Thus  $\G=(G;T_1,T_2;\theta)\in \G_2^{sym}$. Further  $\G'$ may be obtained from $\G$ by a symmetric $2$-tree 0-extension. 
\end{proof}

\begin{claim}
If $v$ is incident to a third edge $e_3$, then $\G'$ is formed from  a symmetric $2$-tree decomposition $\G$ by a symmetric $2$-tree 1-extension.
\end{claim}

\begin{proof}
Without loss of generality, suppose $e_3\in E(T_1')$. For $i\in\{1,2,3\}$, none of the edges $e_i$  are fixed, so they each terminate at some $u_i\in V(G')-{\{v, s_{\theta'}(v)\}}$. Hence each edge $s_{\theta'}(e_i)$ incident to $s_{\theta'}(v)$ terminates at $s_{\theta'}(u_i)\in V(G')-{\{v, s_{\theta'}(v)\}}$.

Since $e_1,e_3\in E(T_1')$ and $T_1'$ is a tree, $u_1u_3\not\in E(T_1')$, and so, since $T_1'$ is symmetric under $\theta'$, $s_{\theta'}(u_1)s_{\theta'}(u_3)\not\in E(T_1')$ either. If $\{u_1, u_3\} = \{s_{\theta'}(u_1), s_{\theta'}(u_3)\}$ then $v, u_1, s_{\theta'}(v), u_3, v$ is a cycle in $T_1'$, which is a contradiction. Hence $\{u_1, u_3\}\neq\{s_{\theta'}(u_1), s_{\theta'}(u_3)\}$ and so we can perform a 2-dimensional 1-reduction at $v$ which adds the edge $u_1 u_3$, followed by a 2-dimensional 1-reduction at $s_{\theta'}(v)$ which adds the edge $s_{\theta'}(u_1)s_{\theta'}(u_3)$  to form the graph $G$. 

Let $\G=(G;T_1,T_2;\theta)$ be the symmetric  $2$-tree decomposition with $G = G'-{\{v, s_{\theta'}(v)\}} + { \{u_1u_3, s_{\theta'}(u_1u_3)\}}$, $T_1 = T_1' - {\{v, s_{\theta'}(v)\}} + { \{u_1u_3, s_{\theta'}(u_1u_3)\}}$, $T_2 = T_2' - {\{v, s_{\theta'}(v)\}}$ and $s_\theta=s_{\theta'}\vert_{V(G)}$. Then $s_\theta(T_i)=T_i$ and so  $\G\in \G_2^{sym}$. Further $T_1$ was obtained from $T_1'$ by a pair of 1-dimensional 1-reductions, and $T_2$ was formed from $T_2'$ by a pair of 1-dimensional 0-reductions. Thus we can reconstruct $\G'$ from $\G$ by a symmetric  $2$-tree 1-extension.
\end{proof}
\end{proof}

This result implies the inductive constructions sought:

\begin{cor}\label{cor:sym_inductive_W}
Let $\G=(G;T_1,T_2;\theta)\in\G_2^{sym}$. 
Then, there exists a sequence of symmetric $2$-tree decompositions in $\G_2^{sym}$,
\[
\K_1= \G^{(1)} \rightarrow \G^{(2)} \rightarrow \cdots \rightarrow \G^{(n)}=\G
\]
such that for all $2\leq i\leq n$, $\G^{(i)}$ is obtained from $\G^{(i-1)}$ by a symmetric $2$-tree $j$-extension, for some $j\in\{0,1\}$.
\end{cor}

\subsection{Realisations with $\mathcal{C}_s$-symmetry}
We shall now show how to construct a symmetric realisation for any symmetric $2$-tree decomposition from the class $\G_2^{sym}$. {We prove this explicitly for realisations under reflection symmetry. The argument for half-turn symmetry is similar.}

 can construct a realisation with either reflectional symmetry or rotational symmetry.

A {\em symmetric placement} of a $\bZ_2$-symmetric multi-graph $(G,\theta)$ in the plane is a pair $(p,\tau)$ consisting of an injective map $p:V(G)\to \mathbb{R}^2$ and a representation $\tau:\bZ_2\to \GL(\mathbb{R}^2)$ such that $\tau(s)(p(v)) = p(s_\theta(v))$ for all $v\in V(G)$. 
If $\tau(s)$ is a reflection in a coordinate axis then we refer to the pair $(p,\tau)$ as a {\em $\C_s$-placement} of $(G,\theta)$.
If  $\tau(s)$ is a half-turn rotation about the origin then we refer to $(p,\tau)$ as a {\em $\C_2$-placement} of $(G,\theta)$.

A {\em $\C_s$-realisation} (respectively, {\em $\C_2$-realisation}) for a symmetric $2$-tree decomposition $\G=(G;T_1,T_2;\theta)$ in the plane is a $\C_s$-placement  (respectively, {\em $\C_2$-placement})  $(p,\tau)$ of $(G,\theta)$ with the property that $(G,p)$  is a realisation for the  $2$-tree decomposition $(G;T_1,T_2)$. 

\begin{prop}\label{prop:InductiveConstruction_Realisation0_Sym}
Let $\G=(G;T_1,T_2;\theta)$ and $\G'=(G';T_1',T_2';\theta')$ be a pair of symmetric $2$-tree decompositions and suppose $\G'$ is obtained by applying a symmetric $2$-tree $0$-extension to $\G$.

If $\G$ has a $\C_s$-realisation $(p,\tau)$ in the plane then $\G'$ has a $\C_s$-realisation $(p',\tau)$ in the plane with the property that $p'(w)=p(w)$ for all $w\in V(G)$.
\end{prop}

\begin{proof}
Let $(p,\tau)$ be a $\C_s$-realisation for $\G$ in the plane.
Suppose the symmetric $2$-tree $0$-extension which forms $\G'$ from $\G$ adjoins the vertices $v$ and $s_{\theta'}(v)$ to $G$. 
Let $\G''$ be the intermediate (and non-symmetric) $2$-tree decomposition obtained by deleting $s_{\theta'}(v)$ and its incident edges from $G'$.
Note that $\G''$ is obtained from $\G$ by a (non-symmetric)
$2$-tree $0$-extension. 
By Proposition \ref{prop:InductiveConstruction_Realisation0}, 
there exists a realisation $p''$ for $\G''$ with the property that $p''(w) = p(w)$ for all $w\in V(G)$. 
Define $p'(w)=p(w)$ for all $w\in V(G)$, $p'(v)=p''(v)$ and 
$p'(s_{\theta'}(v)) = \tau(s)(p'(v))$.	
The resulting pair $(p',\tau)$ is a $\C_s$-realisation for $\G'$.
\end{proof}

To complete the inductive construction we now consider geometric placements for symmetric $2$-tree $1$-extensions.

\begin{prop}\label{prop:InductiveConstruction_Realisation1_Sym}
Let $\G=(G;T_1,T_2;\theta)$ and $\G'=(G';T_1',T_2';\theta')$ be a pair of symmetric $2$-tree decompositions and suppose $\G'$ is obtained by applying a symmetric $2$-tree $1$-extension to $\G$.

If $\G$, and every symmetric $2$-tree decomposition with fewer vertices than $\G$, has a $\C_s$-realisation in the plane then $\G'$ has a $\C_s$-realisation in the plane.
\end{prop}

\begin{proof}
 	Let $(p,\tau)$ be a $\C_s$-realisation for $\G$ in the plane.	We may assume, without loss of generality, that $\tau(s)$ is a reflection in the $y$-axis.
	
	Suppose the symmetric $2$-{tree} $1$-extension which forms $G'$ from $G$ adjoins the vertices $v$ and $s_{\theta'}(v)$. 
Then $d_{G'}(v)=3$, and $G$ is formed from $G'$ by deleting $v$, $s_{\theta'}(v)$, and all edges incident to either of these vertices, before adding an edge $e$ between the vertices in $N_{G'}(v)$ and another edge $s_{\theta'}(e)$ between the vertices in $N_{G'}(s_{\theta'}(v))$. 

Suppose $d_{T_1'}(v)=1$ and $d_{T_2'}(v)=2$. For a pair of vertices $u,w\in V(G)$, write $u\sim^s w$ if either $u\in\{w,s_\theta(w)\}$, or, $u$ is joined to either $w$ or $s_\theta(w)$ by a sequence of parallel edges in $G-\{e,s_\theta(e)\}$.	The construction of $p'$ now follows the proof of Proposition \ref{prop:InductiveConstruction_Realisation1} almost verbatim by replacing $\sim$ with $\sim^s$ and setting 
$p'(s_{\theta'}(v)) = \tau(s)(p'(v))$.
The case where  $d_{T_1'}(v)=2$ and $d_{T_2'}(v)=1$ can be proved by similar methods. 
\end{proof}

\begin{thm}
\label{thm:Symm_CsAxes_ConstructRealisation}
Let $\G=(G;T_1,T_2;\theta)\in\G_2^{sym}$ be a symmetric $2$-tree decomposition with $\G\not=\K_1$. Then $\G$ has a $\mathcal{C}_s$-realisation in the plane.
\end{thm}

\begin{proof}
{ Use Propositions \ref{prop:InductiveConstruction_Realisation0_Sym} and \ref{prop:InductiveConstruction_Realisation1_Sym}, and} apply a similar argument to the proof of Theorem \ref{thm:RealisationIn2Dimensions}.
\end{proof}

 Theorem \ref{thm:Symm_CsAxes_ConstructRealisation} shows that it is always possible to construct examples of isostatic $\C_s$-symmetric frameworks in the $\ell^\infty$ plane which induce prescribed symmetric monochrome spanning trees with no fixed edges. In the following, note that not all $(2,2)$-tight $\bZ_2$-symmetric graphs admit a symmetric $2$-tree decomposition.

\begin{cor}\label{cor:RealisationIn2Dimensions_Z2} 
Let $(G,\theta)$ be a $\bZ_2$-symmetric simple graph. 
If $(G,\theta)$ admits a symmetric $2$-tree decomposition $\G=(G;T_1,T_2;\theta)$, with no fixed edges, then there exists a $\C_s$-realisation for $\G$ in the plane which is well-positioned and minimally rigid in $(\mathbb{R}^2, \|\cdot \|_\infty)$.
\end{cor}

\proof
By Theorem \ref{thm:Symm_CsAxes_ConstructRealisation}, the symmetric $2$-tree decomposition $\G$ has a $\C_s$-realisation  in the plane.
By Theorem \ref{thm:K_MinInfRigid_Simple}, this realisation is  minimally rigid  in $(\mathbb{R}^2, \|\cdot \|_\infty)$.
\endproof

\section{Open problems}\label{sec:OpenProblems} 
{ In Section \ref{sec:2D_Geometry},} we showed that given any multi-graph with a partition of its edge set into two spanning trees, there exists a realisation of this $2$-tree decomposition in the plane. It is not known whether this result extends to $d$-dimensions.
Indeed, a solution here would settle a particular case of another open problem which is to determine whether every $(d,d)$-tight graph has a rigid placement in $(\mathbb{R}^d,\|\cdot\|_q)$ for $d\geq 3$ and $q\not=2$.

\begin{open}[Rigidity for $\ell^q$ norms.]
Let $G$ be a simple graph which is an edge-disjoint union of $d$ spanning trees $T_1,\ldots, T_d$, where $d\geq 3$. 

\begin{enumerate}[(a)] 
\item 
Does there exist a placement of $G$ in $\mathbb{R}^d$ such that the induced monochrome subgraphs of $G$ are precisely $T_1,\ldots,T_d$?

\item
Does there exist an isostatic placement of $G$ in $(\mathbb{R}^d, \|\cdot\|_{q})$ for all (or for some) $q\not=2$? 
\end{enumerate}

A positive answer to (a) would imply a positive answer to (b) in the case $q=\infty$. If $d=2$, then the answer to both questions is ``yes".
 \end{open}

The following example suggests a different approach is required to extend Corollary \ref{cor:RealisationIn2Dimensions} to higher dimensions.

\begin{example}
Suppose $\G=(G;T_1,T_2,T_3)$ is a $3$-tree decomposition with a realisation  $(G,p)$   in $\mathbb{R}^3$.
Suppose $x,y,z$ are vertices of $G$ with $p(x)=(0,0,0)$, $p(y)=(-1,3,-1)$ and $p(z)=(-1,10,-3)$.
Now suppose a $3$-tree $0$-extension is applied to $\G$ at the vertices $x,y,z$, which adds the vertex $v$ and results in a $3$-tree decomposition $\G'=(G';T_1',T_2',T_3')$ with 
$xv\in T_1$, $yv\in T_2$ and $zv\in T_3$. 
For these trees to correspond to the monochromatic trees induced by a framework colouring for $G'$, we must have
\begin{align*}
\|p(v)-p(x)\|_{\infty} &= |v_1-x_1| = |v_1|, \\
\|p(v)-p(y)\|_{\infty} &= |v_2-y_2|= |v_2-3|, \text{ and} \\
\|p(v)-p(z)\|_{\infty} &= |v_3-z_3|=|v_3+3|. 
\end{align*}
However, there is no such point $p(v)=(v_1,v_2,v_3)$ in $\mathbb{R}^3$. 
\end{example}

A related problem is that of constructing examples of redundantly rigid frameworks in $(\mathbb{R}^d, \|\cdot\|_{q})$.  Here a framework is {\em redundantly rigid} if it is rigid and every subframework obtained by the removal of a single edge is also rigid. Such frameworks have played a key role in the study of {\em global rigidity} for frameworks in Euclidean space (see for example \cite{jac-jor}).

\begin{open}[Redundant rigidity for $\ell^q$ norms.]
Let $G$ be a simple graph which is an edge-disjoint union of $d$ Hamilton cycles $H_1,\ldots,H_d$, where $d\geq2$. 

\begin{enumerate}[(a)] 
\item 
Does there exist a placement of $G$ in $\mathbb{R}^d$ such that the induced monochrome subgraphs of $G$ are precisely $H_1,\ldots,H_d$?

\item
Does there exist a redundantly rigid placement of $G$ in $(\mathbb{R}^d, \|\cdot\|_{q})$ for all (or for some) $q\not=2$? 
\end{enumerate}

A positive answer to (a) would imply a positive answer to (b) in the case $q=\infty$. 
\end{open}

Similar realisation problems arise for other norms. For example, it is shown in \cite{KL_MatrixNorms} that rigidity for the cylinder norm on $\bR^3$ is characterised by an induced framework colouring which decomposes the graph into an edge-disjoint union of a spanning tree and a spanning Laman graph. Again whether the existence of such a decomposition implies the existence of a geometric realisation is open.

\begin{open}[Rigidity for the cylinder norm.]
Let $G$ be a simple graph which is an edge-disjoint union of two spanning subgraphs $T$ and $L$ where $T$ is a tree and $L$ is a Laman graph. 

\begin{enumerate}[(a)] 
\item 
Does there exist a placement of $G$ in $(\mathbb{R}^3, \|\cdot\|_{cyl})$ such that the induced monochrome subgraphs of $G$ are precisely $T$ and $L$?

\item
Does there exist an isostatic placement of $G$ in $(\mathbb{R}^3, \|\cdot\|_{cyl})$? 
\end{enumerate}

A positive answer to (a) would imply a positive answer to (b). The smallest graph in this class is $K_6-e$, obtained by removing a single edge from the complete graph $K_6$, and this graph does admit an isostatic placement in $(\mathbb{R}^3, \|\cdot\|_{cyl})$ (see \cite{KL_MatrixNorms}).
\end{open}

Realisation problems of this type also arise in considering {\em forced symmetric rigidity} (see for example \cite{Jordan2016, Malestein2015, Schulze2015InfinitesimalRO} for the Euclidean context). A characterisation is obtained in \cite{KS_Reflection} for forced reflectional symmetry in the $\ell^\infty$ plane which is expressed in terms of framework colourings on the associated gain graphs. In this case, the gain graph is expressed as an edge-disjoint union of a spanning unbalanced map graph and a spanning tree. 

\begin{open}[Forced symmetric rigidity for the $\ell^\infty$ norm.]
Let $G_0$ be a gain graph for a $\bZ_2$-symmetric graph which is expressible as an edge-disjoint union of two spanning subgraphs $T$ and $M$, where $T$ is a tree and $M$ is an unbalanced map graph. 

Does there exist a placement of the covering graph $G$ in the plane, with reflectional symmetry, such that the induced monochrome subgraphs of $G_0$ are precisely $T$ and $M$?
\end{open}

\bibliographystyle{plain}
\bibliography{ConstructingIsostaticFrameworksReferences}

\end{document}